\documentclass[11pt]{amsart}
\usepackage{amsmath,amsfonts,amssymb,amsthm,epsfig,color, tikz}
\usepackage{mathrsfs}
\newcommand{\bb}{\mathbb}

\newcommand{\C}{\bb C}

\newcommand{\Z}{\bb Z}
\newcommand{\R}{\bb R}

\newcommand{\Q}{\bb Q}

\newcommand{\EST}{Erd\"os-Sz\"usz-Tur\'an }

\newcommand{\SL}{\operatorname{SL}}
\newcommand{\GL}{\operatorname{GL}}
\newcommand{\pp}{\mathbf p}
\newcommand{\qq}{\mathbf q}

\newcommand{\xx}{\mathbf x}

\newcommand{\prim}{\operatorname{prim}}

\newcommand{\Id}{\operatorname{Id}}

\newcommand{\Mat}{\operatorname{Mat}}
\newcommand{\est}{\operatorname{EST}}

\newcommand{\M}{\mathcal M}
\newcommand{\om}{\omega}

\newcommand{\q}{\mathcal Q}

\newcommand{\hh}{\mathcal H}

\newcommand{\diag}{\operatorname{diag}}

\newcommand{\La}{\Lambda}
\newtheorem{Theorem}{Theorem}
\numberwithin{Theorem}{section}

\newtheorem{lemma}[Theorem]{Lemma}

\newtheorem*{remark*}{Remarks}
\newtheorem*{question*}{Question}

\newtheorem*{lemma*}{Lemma}

\newtheorem*{theorem*}{Theorem}

\numberwithin{equation}{section}
\begin{document}
\title[The Erd\"os-Sz\"usz-Tur\'an distribution]{The Erd\"os-Sz\"usz-Tur\'an distribution for equivariant processes}
\author{Jayadev S.~Athreya}
\author{Anish Ghosh}
\email{jathreya@uw.edu}
\email{ghosh@math.tifr.res.in}
\address{Department of Mathematics, University of Washington, Padelford Hall, Seattle, WA, 98195}
\address{School of Mathematics, Tata Institute of Fundamental Research, Homi Bhabha Road, Colaba, Mumbai 400005, India}

\thanks{Athreya is supported by National Science Foundation CAREER grant DMS 1351853. Ghosh acknowledges support of a UGC grant, a CEFIPRA grant and a DST Swarnajayanti fellowship. Ghosh acknowledges support of the Benoziyo Endowment Fund for the Advancement of Science at the Weizmann Institute. Parts of this work were completed while both authors were in residence at the Mathematical Sciences Research Institute in Berkeley, California, during the Spring 2015 semester, supported by the National Science Foundation Grant DMS 0932078 000.}
\begin{abstract}
We resolve problems posed by Kesten and Erd\"os-Sz\"usz-Tur\'an on probabilistic Diophantine approximation via methods of homogeneous dynamics. Our methods allows us to generalize the problem to the setting of general measure-valued processes in $\R^n$, and obtain applications to the distribution of point sets which occur in higher-dimensional Diophantine approximation and the geometry of translation surfaces.\end{abstract}
\maketitle
\setcounter{tocdepth}{1}

\small\tableofcontents

\section{Introduction}
\subsection{Dirichlet's Theorem} The most classical result in Diophantine approximation is Dirichlet's Theorem, which is stated in two forms: First, for all irrational $\alpha \in [0, 1]$, and any integer $Q \geq 1$, there is a $1 \le q \le Q$ and a $p \in \Z$ relatively prime to $q$ so that  \begin{equation}\label{eq:dirichlet:1} \left| \alpha q- p \right| \le \frac {1}{Q}.\tag{D1}\end{equation} As a corollary, one obtains that there are infinitely many $\frac{p}{q} \in \Q$ (with $gcd(p, q) =1$) satisfying \begin{equation}\label{eq:dirichlet} \left| \alpha - \frac p q \right| \le \frac {1}{q^2}\tag{D2}.\end{equation} In this paper, we consider the \emph{distribution} of the number of solutions to modified versions of \eqref{eq:dirichlet:1} and \eqref{eq:dirichlet} from a probabilistic perspective initiated by Erd\"os-Sz\"usz-Tur\'an~\cite{EST} and Kesten~\cite{Kesten}. Our methods show how this problem can be generalized and solved in many other geometric and number theoretic settings.

\subsection{The Erd\"os-Sz\"usz-Tur\'an and Kesten distributions}

\noindent In 1958~\cite{EST}, Erd\"os-Sz\"usz-Tur\'an introduced a problem in probabilistic Diophantine approximation: what is the probability $f(N, A, c)$ that a point $\alpha$ chosen from the uniform distribution on $[0,1]$ has a solution $\frac p q \in \Q$ to the modified Dirichlet equation \begin{equation}\label{eq:dirichlet:modified} \left| \alpha - \frac p q \right| \le \frac {A}{q^2} ,\end{equation} with denominator $q \in [N, cN]$? Here $A >0, c>1$ are fixed positive parameters, and $N$ is a parameter which goes to infinity. We note that by a well known result of Hurwitz, $A = \frac{1}{\sqrt{5}}$ is the best allowable constant so that (\ref{eq:dirichlet:modified}) has infinitely many solutions for all $\alpha$.  Given $A, c, N$, let $\est(A,c,N)(\alpha)$ be the number of solutions $p/q \in \Q$ with $\gcd(p, q) = 1$ to (\ref{eq:dirichlet:modified}). Letting $\alpha \in [0, 1]$ be a uniform random variable yields an integer-valued random variable $\est(A,c,N)$, with  $$P(\est(A,c,N) = k ) = m\left(\alpha \in [0, 1]: \mbox{ there are exactly } k \mbox{ solutions to (\ref{eq:dirichlet:modified})} \right),$$ where $m$ is Lebesgue measure on $[0,1]$. Then, the \EST question is the existence of the limit $$\lim_{N \rightarrow \infty} P(\est(A,c, N)>0).$$

\noindent Considering analagously a modified version of~\eqref{eq:dirichlet:1}, Kesten~\cite{Kesten} defined the sequence of random variables $K(A, N)$ (our notation differs from~\cite{Kesten}) as the number of solutions to
\begin{equation}\label{eq:dirichlet:modified:1} \left| \alpha q- p \right| \le \frac {A}{N}, 1 \le q \le N,\end{equation} where $\alpha$ is a uniform $[0,1]$ random variable. That is, $$P(K(A,N) = k ) = m\left(\alpha \in [0, 1]: \mbox{ there are exactly } k \mbox{ solutions to (\ref{eq:dirichlet:modified:1})} \right).$$
We consider the following
\begin{question*} What (if they exist) are the limiting (as $N \rightarrow \infty$) distributions of the random variables $\est(A,c, N)$ and $K(A, N)$?  \end{question*}
%

\noindent Our main results shows that the limiting distribution exists and can be viewed as the probability of a random unimodular lattice intersecting a certain fixed region. Let $X_2 = \SL(2, \R)/\SL(2, \Z)$ denote the space of unimodular lattices in $\R^2$,via the identification $$g \SL(2, \Z) \mapsto g \Z^2.$$ Let $\mu_2$ denote the Haar probability measure on $X_2$, and given $\La \in X_2, \La = g \Z^2$, let $\La_{\prim}$ be the set of primitive vectors in $\La$.

\begin{Theorem}\label{theorem:est:lattice} The limiting distribution of the random variables $\est(A,c,N)$ and $K(A, N)$ exist and denoting the random variables with these limiting distributions as $\est(A,c)$ and $K(A)$, we have \begin{equation}\label{eq:est:dist}P(\est(A,c) = k) = \mu_2(\La \in X_2: \#(\La_{\prim} \cap H_{A,c}) = k ),\end{equation} and 
\begin{equation}\label{eq:kesten:dist}P(K(A) = k) = \mu_2(\La \in X_2: \#(\La_{\prim} \cap R_{A}) = k )
\end{equation} where
\begin{equation}\label{def:region}
H_{A, c} = \{(x,y) \in \R^2~:~xy \leq A, 1 \leq y \leq c\},
\end{equation}
and
\begin{equation}\label{def:region:rectangle}
R_{A} = \{(x,y) \in \R^2~:~|x| \leq A, 0 \leq y \leq 1\}.
\end{equation}

\end{Theorem}





\noindent Expressing the limiting distributions as distributions on the space of lattices allows us to apply classical results on the geometry of numbers to obtain moment and concentration estimates (Theorem~\ref{variance}). Furthermore, our translation of the problem into a geometric and dynamical problem is axiomatic and flexible, and as we will see in \S\ref{sec:generalized} applies to point sets associated to linear forms and translation surfaces. Our proof also shows that the same result holds as long as $\alpha$ is chosen from a probability measure with a continuous density (and in higher dimensions, certain natural classes of singular measures). Our main result (Theorem~\ref{theorem:est}) describes how to define, construct and compute related distributions in the general setting of equivariant processes, which we define precisely in \S\ref{sec:equivariant}. We note that a function very similar to the one considered in (\ref{eq:kesten:dist}) above has been computed explicitly in the paper \cite{SV} by Str\"{o}mbergsson and Venkatesh, and indeed an adaptation of their methods might allow for an explicit computation here as well.

\subsection{History}  This circle of problems has a long history, starting with the original paper of \EST~\cite{EST}. There, in addition to posing the problem of the limiting probability of $P(\est(A, c, N)>0)$ (they denoted this putative limit as $f(A, c)$), they showed that for $A \le \frac{c}{1+c^2}$, that the limit existed and $$f(A, c) = \frac{12}{\pi^2} A \log c.$$ Subsequently Kesten~\cite{Kesten} showed that the the limit exists under the assumption $Ac \le 1,$ and this assumption was removed by Kesten-Sos~\cite{KestenSos}.

We were introduced to this problem by Boca, Zaharescu, and Heersink (who has studied the problem of finding approximates with appropriate congruence conditions~\cite{Heersink}). Their methods, number theoretic in nature, yield explicit formulas for $f(A, c)$ (computed independently by Boca~\cite{Boca} and Xiong-Zaharescu~\cite{XZ}), and they also considered localizing $\alpha$ to smaller intervals.

Our results offer detailed information on the distributions of the random variables $\est(A,c)$ and $K(A)$, yielding the limiting distribution (not just the probability of positivity) of both random variables in both this setting and in a variety of other geometric and number theoretic contexts, and our methods can also localize $\alpha$ to even \emph{shrinking} intervals (\S\ref{sec:measures}).

Our results on the Kesten distribution were anticipated by Marklof~\cite[Theorem 4.4]{Marklof}, thought at the time he was not aware of Kesten's question. Marklof-Str\"ombergsson~\cite{MS, MSquasi} have given several beautiful applications of homogeneous dynamics to to statistical physics via understanding the fine scale statistics of point sets.

\subsection{Organization} In the remainder of this introduction, we prove Theorem~\ref{theorem:est:lattice}, using equidistribution of horocycles on the space $X_2$. {In \S\ref{sec:generalized}, we describe our general results in the settings of lattices; linear forms; diophantine approximation on curves, and translation surfaces. In \S\ref{sec:equivariant} we describe our general philosophy and state our main axiomatic theorem.  In \S\ref{sec:equilattices}, we prove our results in their various incarnations on the space of lattices, and in \S\ref{sec:translation}, we prove our results in the setting of translation surfaces. 

\noindent\textbf{Acknowledgments.} We thank the Mathematical Sciences Research Institute (MSRI), Berkeley for ideal working conditions. We thank Florin Boca, Byron Heersink, and Alexandru Zaharescu for introducing us to this problem and inspiring this project. J.S.A. would like to thank the Tata Institute of Fundamental Research for its hospitality in August 2015 and February 2016, and A.G. thanks the University of Washington for its hospitality in May 2016. We thank the referees for useful suggestions which have improved the exposition.
\subsection{Equidistribution on the modular surface}\label{sec:equi}

\noindent We relate the Erd\"os-Sz\"usz-Tur\'an and Kesten distributions to dynamics on the space of unimodular lattices, and prove Theorem~\ref{theorem:est:lattice}. This approach will generalize to higher dimensions, and will allow us derive a number of results using appropriate equidistribution results.  We begin with the original question of Erd\"os-Sz\"usz-Tur\'an and Kesten, proving Theorem~\ref{theorem:est:lattice}.


\subsubsection{Proof of Theorem~\ref{theorem:est:lattice}}We note $\est(A,c,N) = k$ if and only if $$\exists \mbox{ exactly }k \mbox{ distinct } \frac p q \in \Q \mbox{ such that } N \le q \le cN \mbox{ and } \left|\alpha - \frac p q \right| < \frac{A}{q^2}.$$ Equivalently, there are exactly $k$ vectors $$\begin{pmatrix}p\\q\end{pmatrix} \in \Z^2_{\prim} \mbox { such that }  \begin{pmatrix}x\\y\end{pmatrix} := u_{\alpha} \begin{pmatrix}p\\q\end{pmatrix} \in g_{\log N} H_{A, c},$$
where $$u_{\alpha} = \begin{pmatrix} 1 & -\alpha\\ 0 & 1 \end{pmatrix} \mbox{ and } g_t = \begin{pmatrix} e^{t} & 0\\ 0 & e^{-t}\end{pmatrix}.$$ This follows by rewriting $$ \left|\alpha - \frac p q \right| < \frac{c}{q^2}, N \le q \le cN $$ as $$ q|q\alpha-p| < A,$$ and then as $$|xy| < A, N \le y \le cN.$$
\noindent Thus, we are interested in the measure of the set of $\alpha \in [0, 1]$ satisfying
\begin{equation}\label{eq:Th1:2}
 \#(g_{\log N}u_{\alpha}\Z^{2}_{\prim} \cap H_{A, c} )= k.
\end{equation}

\noindent Let $\chi_{k}$ denote the indicator function of the set
\begin{equation}\label{unbounded}
\left\{\Lambda \in X~:~ \#\left(\Lambda_{\prim} \cap H_{A, c}\right)=k\right\}.
\end{equation}
To compute $P(\est(A,c) = k)$, we are interested in the $N \rightarrow \infty$ behavior of \begin{equation}\label{eq:chik}\int_{0}^{1}\chi_{k}(g_{\log N}u_{\alpha}\Z^{2}_{\prim})d\alpha.\end{equation}

\noindent  Let $\eta_N$ denote the measure $d\alpha$ on the set $\{g_{\log N}u_{\alpha}\Z^{2}_{\prim}: 0 \le \alpha \le 1\}$, so  we can rewrite (\ref{eq:chik}) as $\eta_N (\chi_K)$.  We will now use suitable approximations of the functions $\chi_{k}$ and apply Zagier's equidistribution theorem~\cite[p. 279]{Z}, which would tell us that $$\eta_N \longrightarrow \mu_2,$$ as $N \rightarrow \infty$, where the convergence is in the weak-$*$ topology. This gives
$$P(\est(A,c, N) = k)  = \eta_N(\chi_K) \underset{N \to \infty}{\longrightarrow} \mu_2(\chi_K).$$
It is important to note that this approximation process is complicated by the fact that the sets (\ref{unbounded}) are unbounded, so we cannot approximate $\chi_{k}$ using functions of compact support. However, the sets can be approximated from above and below by bounded continuous functions using that the boundary of the support of $\chi_{k}$ has measure zero with respect to the Haar measure $\mu_2$ on $X_2$. We refer the reader to \cite{MS} which has a more general form of this argument. Indeed, as pointed out to us by an anonymous referee, one can apply (6.28) in Theorem 6.7 of loc. cit., with $\alpha=(0,0), m=1, M=I$ in $\SL(2, \R)$, and for a suitable choice of set $\mathfrak{B}_t^{(i)}$ defined by (\ref{unbounded}) independent of $t$.

For the Kesten distribution, we note by a similar argument that $K(A, N) = k$ if and only if $$ \#\left(g_{\log N}u_{\alpha}\Z^{2}_{\prim} \cap R_{A} \right) = k.$$ Thus, proceeding as above, we obtain \eqref{eq:kesten:dist}. 
\qed\medskip

\subsubsection{Measures and windows}\label{sec:measures} The proof of Theorem~\ref{theorem:est:lattice} in fact yields much more information. A strengthening of Zagier's theorem due to Shah~\cite{Shah} allows us to obtain the equidistribution result for any absolutely continuous measure on $[0,1]$. Thus, the limiting random variables $\est$ and $K$ do not depend on the initial distribution of $\alpha$ (as long as it is continuous). A different strengthening of Zagier's result is due to Hejhal~\cite{Hejhal} (with subsequent work of Strombergsson~\cite{Strombergsson}), which implies that we can sample $\alpha$ from smaller subintervals depending on $N$, as long as the subintervals shrink no faster than $N^{-1/2}$.

We can also consider the \EST and Kesten distributions associated to solutions of \eqref{eq:dirichlet:modified} and \eqref{eq:dirichlet:modified:1} with $c_1N \le q \le c_2N$, with $0<c_1<c_2$. The limiting distributions will again be given by the probability random lattices intersect fixed regions in $k$ points, with the regions being given by $$H_{A,c_1,c_2} =  \{(x,y) \in \R^2~:~xy \leq A, c_1 \leq y \leq c_2\}$$ and $$R_{A, c_1,c_2} = \{(x,y) \in \R^2~:~|x| \leq A, c_1 \leq y \leq c_2\}.$$

\subsubsection{Moments and concentration}\label{sec:momentsexample} To compute moments of the \EST and Kesten distribution, we need to understand the quantities $$\sum_{k=0}^{\infty} k^t P(X=k), t \in \R,$$ where $X$ is either $\est(A,c)$ or $K(A)$. For $t=1$, we can rewrite this as $$\int_{X_2}  \#\left(\La_{\prim} \cap H_{A,c}\right) d\mu_2(\La).$$ By the Siegel mean value theorem~\cite{Siegel}, we have $$\int_{X_2}  \#\left\{\La_{\prim} \cap H_{A,c}\right\} d\mu_2(\La) = \frac{6}{\pi^2} \left| H_{A,c}\right| = \frac{12}{\pi^2} A \log C.$$ This is the \emph{expected} number of solutions to (\ref{eq:Th1:2}) (in the $N\rightarrow\infty$ limit). Note that if $A \le \frac{c}{1+c^2}$, $$P(\est(A,c)>0) = \sum_{k \geq 1} P(\est(A,c)=k) = \frac{12}{\pi^2} A \log C,$$ so we have $$ E(\est(A,c)) = \sum_{k \geq 1} kP(\est(A,c) =k) = \sum_{k \geq 1} P(\est(A,c)=k),$$ so for $k>1$, $P(\est(A,c) =k)=0$, that is, the \EST distribution is concentrated at $1$. For the Kesten distribution, similar computations yield the mean, $$E(K(A)) = \frac{6}{\pi^2} \left|R_A\right| = \frac{6A}{\pi^2}.$$ In the setting of higher-dimensional Diophantine approximation, we will obtain bounds on higher moments via classical results on the geometry of numbers.

\section{Erd\"os-Sz\"usz-Tur\'an and Kesten distributions in higher dimensions}\label{sec:generalized}

\subsection{Diophantine approximation} We start with a natural higher dimensional generalization of the original \EST problem. Let $d \geq 2 = m + 1$ ($d = 2, m = 1$ corresponds to our original problem), and fix $A>0, c>1$ and a norm $\|\cdot\|$ on $\R^{m}$. Let $\bf{x}$ be chosen from the uniform distribution on $[0,1]^m$, and let $\est_d(A, c, N)$ denote the number of solutions $(\pp, q) \in \Z^m \times \Z$ (with $(\pp, q)$ primitive) to the modified Dirichlet equation \begin{equation}\label{eq:dirichlet:higher} \left\| \mathbf{x}q - \pp \right\| \le Aq^{-\frac 1 d},\end{equation} with $q \in [N, cN]$, and $K_d(A, N)$ denote the number of solutions to \begin{equation}\label{eq:dirichlet1:higher} \left\| \mathbf{x}q - \pp \right\| \le AN^{-\frac 1 d},\end{equation} with $q \in [1, N]$.

\begin{Theorem}\label{theorem:est:lattice:d} The limiting distributions of $\est_d(A,c,N)$ and $K_d(A,N)$ exist and the distributions of the limiting random variables $\est_d(A,c)$ and $K_d(A)$ are given by  $$P(\est_d(A, c) ) = \mu_d(\La \in X_d: \#\left(\La_{\prim} \cap H_{d,A,c}\right)= k )$$ and $$P(K_d(A) ) = \mu_d(\La \in X_d: \#\left(\La_{\prim} \cap R_{d,A}\right) = k )$$where\begin{equation}\label{def:region:d}
H_{d, A, c} = \{(\xx,y) \in \R^d \times \R~:~\|\xx\|y \leq A, 1 \leq y \leq c\}
\end{equation}
and \begin{equation}\label{def:region:d:rectangle}
R_{d, A} = \{(\xx,y) \in \R^d \times \R~:~\|\xx\| \leq A, 0 \leq y \leq 1\}
\end{equation}\end{Theorem}

\subsection{Linear Forms}\label{sec:linear} Next, we consider systems of linear forms. Let $d = m+n$, $m, n \geq 1$, fix $A>0, c>1$, and norms $\|\cdot\|_m$ and $\|\cdot\|_n$ on $\R^m$ and $\R^n$. We consider the set of $m$ linear forms in $n$ variables, parameterized by $M_{m\times n}(\R)$, the set of $m \times n$ real matrices. We identify $M_{m \times n}(\R)$ with $\R^{mn}$. Let $X$ be chosen from the uniform distribution on $[0,1]^{mn}$. We define the random variable $\est_{m\times n}(A, c, N)$ as the number of solutions $(\pp, \qq) \in \Z^m \times \Z^n$ (with $(\pp, \qq)$ primitive) to the modified Dirichlet equation \begin{equation}\label{eq:dirichlet:linear} \left\| X\qq -\pp \right\|_m \le A\|\qq\|_n^{-\frac n m},\end{equation} with $\|q\|_n \in [N, cN]$. Similarly, we define $K_{m\times n}(A,N)$ as the number of solutions to  \begin{equation}\label{eq:dirichlet1:linear} \left\| X\qq -\pp \right\|_m \le A|N|^{-\frac n m},\end{equation} with $\|q\|_n \in [1, N]$.

\begin{Theorem}\label{theorem:est:linear} The limiting distributions of the random variables $\est_{m \times n}(A,c, N)$ and $K_{m \times n}(A, N)$ exist and the distributions of the limiting random variables $\est_{m\times n} (A, c)$ and $K_{m \times n} (A)$ are given by  $$P(\est_{m \times n}(A, c)=k) = \mu_d(\La \in X_d: \#\left(\La_{\prim} \cap H_{m\times n,A,c}\right) = k )$$and $$ P(K_{m \times n}(A)=k) = \mu_d(\La \in X_d: \#\left(\La_{\prim} \cap R_{m\times n,A}\right) = k )$$

\begin{equation}\label{def:region:mn}
H_{m\times n, A, c} = \{(\xx,\mathbf y) \in \R^m \times \R^n~:~\|\xx\|_m\|\mathbf y\|_n \leq A, 1 \leq \|\mathbf y\|_n \leq c\}.
\end{equation}

\begin{equation}\label{def:region:mn:rectangle}
R_{m\times n, A} = \{(\xx,\mathbf y) \in \R^m \times \R^n~:~\|\xx\|_m \leq A, 0 \leq \|\mathbf y\|_n \leq 1\}.
\end{equation}

\end{Theorem}

\noindent We note that Theorem \ref{theorem:est:lattice:d} is a special case of the above Theorem with $n = 1$.

\subsection{Approximation on Curves} \noindent The subject of metric Diophantine approximation \emph{on manifolds} studies typical Diophantine properties of points on manifolds. It is well known and easy to see using the Borel Cantelli Lemma, that almost every real number is \emph{not very well approximable}. This means that the inequality
$$ |qx - p| < 1/|q|^{1+\epsilon} $$
\noindent has at most finitely many solutions. This result generalises easily to arbitrary dimension. In $1932$, K. Mahler conjectured that almost every point on the curve $(x, x^2, \dots, x^n)$ is not very well approximable. Mahler's conjecture started the subject and there have been many subsequent works, including recent dramatic advances due to Kleinbock-Margulis, Beresnevich, Velani, and others. The constraint of lying on a manifold makes the subject considerably more complicated than classical Diophantine approximation. Nevertheless, our approach can be used to compute \EST and Kesten distributions for vectors lying on curves. Let $d=n+1$, and $\phi : [a, b] \to \R^{n}$ be an analytic curve whose image is not contained in a proper affine subspace, and $\|\cdot\|$ denote a norm on $\R^{n-1}$. Let $x$ be chosen from the uniform distribution on $[a,b]$, and let $\est_{\phi}(A,c)$ denote the random variable counting solutions to $$\|q\phi(x) - \mathbf p\| < Aq^{-\frac 1 d}, N \leq q \leq cN.$$ Let $K_{\phi}(A,N)$ denote the random variable counting solutions to $$\|q\phi(x) - \mathbf p\| < AN^{-\frac 1 d}, 1 \leq q \leq N.$$

\begin{Theorem}\label{Th:main2} The random variables $\est_{\phi}(A,c,N)$ and $K_{\phi}(A,N)$ have limiting distributions, and the limiting random variables $\est_{\phi}(A,c)$ and $K_{\phi}(A)$ have distributions given by $$P(\est_{\phi}(A,c) =k) = \ \mu\left\{\Lambda \in X_d~:~ \#(\Lambda_{\prim} \cap H_{d, A,c}) = k\right\},
$$ and $$P(K_{\phi}(A,c) =k) = \ \mu\left\{\Lambda \in X_d~:~ \#(\Lambda_{\prim} \cap R_{d, A}) = k\right\},
$$ where $H_{d,A,c}$ and $R_{d, A}$ are as in Theorem~\ref{theorem:est:lattice:d}
\end{Theorem}
\medskip
\noindent
\textbf{Remark:} There is also an analogue of Theorem~\ref{theorem:est:linear} for curves in the space of linear $\R^{mn}$, which is exactly parallel to Theorem~\ref{Th:main2}.

\subsection{Measures and windows} As in the setting of $1$-dimensional approximation, we can also work with $q$ (or $\|q\|$) in appropriate subranges of the form $[c_1N, c_2N]$ with appropriate changes to the limiting distributions (replacing the $y$ range with $[c_1, c_2]$). Additionally, choosing absolutely continuous measures also does not change the limiting distribution.

\subsection{Moments and concentration}\label{sec:moments} Classical results from the geometry of numbers allow us to compute moments of the random variables $\est$ and $K$. We recall the definition of the \emph{Siegel transform}: given $f \in C_c(\R^d)$ and $\La \in X_d$ define $$\widehat{f}(\La) = \sum_{v \in \La_{\prim}} f(v).$$ Siegel showed $$\int_{X_d} \widehat{f} d\mu_d = \frac{1}{\zeta(d)} \int_{\R^d} fdm,$$ where $f$ is Lebesgue measure. Thus, the expectation of the random variables $\est$ and $K$ is proportional to the volume of the regions $H_d(A,c)$ and $R_d(A)$. $R_{d}(A)$ grows polynomially in $A$ and $H_{d}(A, c)$ polynomially in $c$ and logarithmically in $A$.

Building on Siegel's work, Rogers~\cite{Rogers, Rogers:sets} and Schmidt~\cite{Schmidt} computed bounds for higher moments of $\widehat{f}$.  These can be exploited to give precise moment estimates for $\est$ and $K$, and yield non-trivial concentration phenomenon. For example, a consequence of~\cite[Lemma 4]{Rogers:sets} shows that for integers $p< d$, $\widehat{f} \in L^p(\mu_d)$, and moreover $$\| \widehat{f}\|_p^p \le  \|f\|_1^{p} + C_{p,d} \|f\|_1^{p-1}.$$ Let $$M_{p,d}(f) := \left\|\widehat{f} - \frac{1}{\zeta(d)} \int_{\R^d} fdm\right\|^p_p.$$ We have $$\mu\left(\La \in X_d: \left|\widehat{f} - \frac{1}{\zeta(d)} \int_{\R^d} fdm\right| > T\right) <\frac{M_{p,d}(f)}{T^p}.$$
Specializing to the case $p=2$:

\begin{Theorem}\label{variance}(Rogers~\cite[Theorem 4]{Rogers}, Schmidt~\cite[Theorem 3]{Schmidt}, see also~\cite[Lemma 4.3]{AMarg}) Let $X$ be either $\est_{m \times n}(A, c)$ or $K_{m \times n}(A)$. There is a constant $C_d$, depending only on dimension $d = m+n$, so that $$V(X)= E((X-E(X))^2)  \le C_d E(X).$$ In particular, for any $T>0$, $$P\left( \left|X - E(X)\right| > T\sqrt{E(X)}\right) \le \frac{C_d}{T^2}.$$

\end{Theorem}

\begin{proof} The second assertion is an immediate consequence of the first. For the first, we note both $\est_{m \times n}$ and $K_{m \times n}$ are random variables counting the number of lattice points in a bounded set ($H_{m \times n}(A, c)$ and $R_{m \times n}(A)$). By~\cite[Lemma 4.3]{AMarg} (which is essentially contained in Rogers), we have that for any random variables of this type, $$E(X^2) \le \mu^2 + C_d \mu.$$ In fact $C_d$ can be chosen to be $8 \zeta(d-1)/\zeta(d)$ (for $d \geq 3$).

\end{proof}

\medskip

\noindent\textbf{Remarks:} Kesten~\cite[Theorem 3]{Kesten} considered the $d \rightarrow \infty$ limit of $K_d$ and proved Poisson behavior (under appropriate normalizations) using the method of moments.

\section{Translation Surfaces}\label{sec:translation} Our approach also yields information on the geometry of the set of holonomy vectors of saddle connections on translation surfaces. Given $g \geq 1$, an \emph{translation surface} $S$ of genus $g$ is a pair $S= (X, \omega)$, where $X$ is a compact Riemann surface of genus $g$ and $\omega$ is a holomorphic $1$-form. A \emph{saddle connection} $\gamma$ on $S$ is a geodesic (in the flat metric determined by $\omega$) connecting two zeros of $\omega$, with none in its interior. The holonomy vector of $\gamma$ is defined by $$z_{\gamma} : = \int_{\gamma} \omega \in \C.$$ The set $$\La_{S} : = \{z_{\gamma} : \gamma \mbox{ a saddle connection on } S\}$$ is a discrete subset of $\R^2$ with quadratic growth (cf. Masur~\cite{Masur}), that is there are constants $0 < c_1 \le c_2$ so that $$c_1 R^2 \le \#(\La_S \cap B(0, R)) \le c_2 R^2.$$

We define the moduli space $\Omega_g$ of translation surfaces by considering equivalence classes of translation surfaces up to biholomorphism. This space is decomposed into strata $\hh(\alpha)$ consisting of holomorphic differentials with zeros of order $\alpha_1, \ldots, \alpha_k$, where $\alpha = (\alpha_1, \ldots, \alpha_k)$ is an integer partition of $2g-2$. Each stratum consists of at most 3 connected components~\cite{KZ}, and there is a natural Lebesgue probability measure $\mu_{\hh}$ on the each stratum $\hh(\alpha)$, known as Masur-Veech measure.

There is a natural $\SL(2, \R)$-action on the space $\Omega_g$ which respects the decomposition into strata, and acts ergodically on each connected component of a stratum~\cite{Masur:IET, Veech}. The set $\La_S$ varies equivariantly under this action, that is $$\La_{gS} = g\La_S,$$ where $\SL(2,\R)$ acts on $\R^2$ by the usual linear action. The fine-scale geometry of the sets $\La_S$ has been a subject of much recent investigation~\cite{AChaika, AChaikaLelievre, SmillieWeiss, UyanikWork}, and our approach allows us to define \EST and Kesten distributions associated to translation surfaces. We note that for $g=1$, $\Omega_1 = X_2$, so this setting is another natural generalization of the original \EST and Kesten problems.

Let $\theta \in [0, 2\pi)$ be chosen from the uniform distribution. We want to understand how well vectors in $\La_S$ approximate the direction $\theta$, in terms of their length. Given $A>0, c>1, N>0$, define the random variables $\est(S, N)$ and $K(S, N)$ by $$\est(S, N) = \# \left(r_{\theta} \La_S \cap H_{A, c, N}\right)$$ and $$K(S, N) = \# \left(r_{\theta} \La_S \cap R_{A, N}\right),$$ where $$H_{A, c, N} = \{(x,y) \in \R^2~:~xy \leq A, N \leq y \leq cN\}$$
and $$R_{A, N} = \{(x,y) \in \R^2~:~|x| \leq A, 0 \leq y \leq N\}.$$

\noindent We say $S_0 \in \Omega_g$ has \emph{circle limit measure} $\mu$ on $\Omega_g$ if the measures $d\theta$ on $\{g_t r_{\theta} S\}_{0 \le \theta < 2\pi}$ converge to $\mu$. A result of Nevo\footnote{The authors have been informed by A. Nevo that this result  was in circulation for a long time but was not published. Now, a paper dealing with the $\SL(2, \R)$ case is availablev \cite{Nevo} and is soon going to be supplanted by an article which deals with the case of general semisimple groups}, see for instance the paper \cite{EskinMasur} of Eskin and Masur, shows that for any stratum $\hh$, $\mu_{\hh}$-a.e. $S \in \hh$ has circle limit measure $\mu_{\hh}$.

\begin{Theorem}\label{theorem:estk:flat} Suppose $S_0 \in \Omega_g$ has circle limit measure $\mu_0$. Then $\est(S_0, N)$ and $K(S_0, N)$ both have limiting distributions, and denoting the random variables with this limiting distribution by $\est(S)$ and $K(S)$, we have $$P(\est(S_0) = k) = \mu_0\left( S \in \Omega_g: \#\left(\La_S \cap H_{A, c} = k\right)\right)$$ and $$P(K(S_0) = k) = \mu_0\left( S \in \Omega_g: \#\left(\La_S \cap R_{A} = k\right)\right).$$ In particular, for any stratum $\hh$ and $\mu_{\hh}$-a.e. $S_0 \in \hh$, $$P(\est(S_0) = k) = \mu_{\hh}\left( S \in \hh: \#\left(\La_S \cap H_{A, c}\right)= k\right)$$ and $$P(K(S_0) = k) = \mu_{\hh}\left(S \in \hh: \#\left(\La_S \cap R_{A} \right)= k\right).$$

\end{Theorem}

\subsection{Lattice Surfaces}
For particular highly symmetric surfaces, we can say more. We denote the stabilizer of the point $S = (X, \omega) \in \Omega_g$ under the $\SL(2, \R)$ action by $\SL(X, \omega)$. A translation surface $S$ is called a \emph{lattice surface} (also known as an \emph{Veech surface}) if $\SL(X, \omega)$ is a lattice. The lattices that occur are always nonuniform, and the $\SL(2, \R)$ orbit of $S$ is closed, a copy of $\SL(2, \R)/\SL(X, \omega)$ in $\Omega_g$. For these surfaces, we have

\begin{Theorem}\label{theorem:estk:veech} Suppose $S_0 = (X_0, \omega_0)$ is a lattice surface, and write $\Gamma = \SL(X_0, \omega_0)$. Let $\mu_{\Gamma}$ denote the Haar probability measure on $\SL(2,\R)/\Gamma$. Then $\est(S_0, N)$ and $K(S_0, N)$ both have limiting distributions, and denoting the random variables with this limiting distribution by $\est(S_0)$ and $K(S_0)$, we have $$P(\est(S_0) = k) = \mu_{\Gamma} \left( g\Gamma \in \SL(2,\R)/\Gamma: \#\left(g\La_{S_0} \cap H_{A, c} = k\right)\right)$$ and $$P(K(S_0) = k) = \mu_{\Gamma}\left( g\Gamma \in \SL(2,\R)/\Gamma: \#\left(g\La_{S_0} \cap R_{A} = k\right)\right).$$

\end{Theorem}

\subsection{Expectation} To compute the expectation of the random variables $K$ and $\est$ in this setting, we use the \emph{Siegel-Veech} formula~\cite{Veech:siegel}. This states that for any $\SL(2, \R)$-invariant measure $\mu$ on $\hh$ where the \emph{Siegel-Veech transform}$$\widehat{f}(S) = \sum_{z \in \La_S} f(z)$$ is in $L^1(\mu)$ for any $f \in C_c(\R^2)$, there is a constant (the \emph{Siegel-Veech constant}) $c_{\mu}$ so that $$ \int_{\hh} \widehat{f}(S) = c_{\mu} \int_{\R^2} f dm,$$ where $m$ is Lebesgue measure on $\R^2$. Applying this to our situation, we say that the expectations of our limiting random variables is given by a scalar multiple of the area of the sets $H_{A,c}$ and $R_{A}$, depending on the circle limit measure. The computation of Siegel-Veech constants is an active and challenging area of research, see, for example~\cite{EMZ} for seminal work. In the setting of lattice surfaces, Veech~\cite{Veech:siegel} related these constants to the covolume of $\SL(X, \om)$ in $\SL(2, \R)$.

\section{Equivariant Processes}\label{sec:equivariant} In this section we define the axiomatic setup of equivariant measure-valued processes and state our main result Theorem~\ref{theorem:est}. This perspective is inspired by the work of W.~Veech~\cite{Veech:siegel} and J.~Marklof, as well as that of A.~Eskin and H.~Masur~\cite{EskinMasur}. It has the great advantage that once we make the proper definitions, the proof of the main theorems are essentially tautologies. The power of the method lies in its flexibility: we will see that the axioms can be verified in several different situations.

\subsection{Equivariant measure processes} Let $n \geq 2$, and $G \subset \GL(d, \R)$. Let $(X, \mu)$ denote a Borel-$G$-space together with a $G$-invariant Borel probability measure $\mu$. A ($G$-)\emph{equivariant measure process} (also known as a  \emph{Siegel measure}, see~\cite{Veech:siegel}) is a triple $(X, \mu, \nu)$ where $\nu$ is a map $$\nu: X \rightarrow \M(\R^d)$$ from $X$ to the space $\M(\R^d)$ of $\sigma$-finite Radon Borel measures on $\R^d$ satisfying the equivariance condition $$\nu(gx) = g_*\nu(x)$$ for all $g \in G, x \in X$, where $G$ acts linearly on $\R^d$.

\subsection{Erd\"os-Sz\"usz-Tur\'an distributions} Given a sequence of equivariant measure processes $\mathbb X = \{(X_N, \eta_N, \nu_N)\}$ and a Borel subset $\mathcal R \subset \R^d$
 we define the \emph{Erd\"os-Sz\"usz-Tur\'an distribution} $\eta = \eta(\mathbb X, \mathcal R)$ on $\R^+$ as the measure given by (if the the limit exists) $$\eta(\bb X, \mathcal R)(0, t) = \lim_{N \rightarrow \infty} \eta_N(x \in X: \nu_N(x)(\mathcal R) \le t).$$

\subsection{Equidistribution} Our main result concerns the setting where our sequence $\mathbb X = \{ X, \eta_N, \nu\}$, that is, a sequence of measures $\eta_n$ on a fixed $G$-space $X$ together with an assignment $\nu$.

\begin{Theorem}\label{theorem:est} Suppose $\eta_N \rightarrow \mu$ (in the weak-* topology). Then $$\eta(\mathbb X, \mathcal R) (0, t) = \mu(x \in X: \nu(x)(\mathcal R) \le t).$$

\end{Theorem}

\begin{proof} By assumption, our measures are all Radon Borel measures, so if $\eta_N \rightarrow \mu$ in the weak-* topology, we have convergence in measure for bounded continuous functions. Now using an approximation argument to approximate Borel subsets from above and below as before gives that for any (fixed) Borel measurable subset $B \subset X$, $$\eta_N(B) \longrightarrow \mu(B).$$ Applying this to $B = \{ x \in X: \nu(x)(\mathcal R) \le t\}$, we have our result.
\end{proof}

This theorem, as stated, is a tautology. The key to applying it is finding appropriate equidistribution results that allow one to take a natural sequence $\mathbb X$ and find a limiting measure so that $\eta_N \rightarrow \mu$.

\subsection{Orbits and point processes} In many of our applications, the measures $\eta_N$ will be supported on orbits of subgroups $H \subset G$ and be the push-forward of some measure on $H$ under the orbit map. In addition, $\nu$ will often be a \emph{point process}, that is, the assignment of a discrete set with counting measure.

\subsection{Applications} In this paper we focus on the applications of this formalism in the space of lattices and the space of translation surfaces. In~\cite{ABG}, we apply these ideas to the Clifford plane in order to understand cusp excursions on general hyperbolic manifolds.

\section{Equidistribution on the space of lattices}\label{sec:equilattices}
\noindent We prove our main Diophantine results using equidistribution results for flows on the space of unimodular lattices. Let  $\mu_d$ denote the Haar probability measure on $$X_d = \SL(d, \R)/\SL(d, \Z).$$ $X_d$ is the space of unimodular (covolume $1$) lattices in $\R^d$, via the identification $$g\SL(d, \Z) \longmapsto g \Z^d.$$ Given $\Lambda = g\Z^d \in X$, we say that $v \in \Lambda$ is \emph{primitive} if $$v = gw, w \in \Z^{d} \backslash \{0\}, \gcd(w) = 1$$ and denote by $\Lambda_{\prim}$ the set of primitive points in $\Lambda$. For all of our Diophantine results, we will use the equivariant assignment $$g\SL(d, \Z) \longmapsto \sum_{v \in g\Z^d_{\prim}} \delta_v ,$$ which, in the notation of \S\ref{sec:equivariant} we view as a map $\nu_d: X_d \rightarrow \mathcal M (\R^d).$ Let $m, n$ be positive integers and let $d = m+n$. Set

$$G = \SL(d, \R), \Gamma = \SL(d, \Z), u_{X} = \begin{pmatrix} \Id_m & X\\ 0 & \Id_n\end{pmatrix}, H = \{u_X~:~X\in \Mat_{m \times n}(\R)\}. $$

\noindent The group $H$ is the expanding horospherical subgroup of $G$ with respect to
\begin{equation}\label{def:diag}
g_t = \diag(e^{t/m}, \dots, e^{t/m}, e^{t/n}, \dots, e^{t/n}), t > 0.
\end{equation}

\noindent The following Lemma is a straightforward generalisation of the the argument in the introduction and allows us to interpret the \EST and Kesten distributions in terms of homogeneous dynamics.

\begin{lemma}
Let notation be as above. Then
$$ \est(A,c,N) = k \text{ if and only if }  \#\left(g_{\log N}u_{X}\Z^{d}_{\prim} \cap H_{A, c} \right) = k.
$$
\end{lemma}

\noindent We can therefore proceed as before. Let $\eta_N$ denote the measure $dY$ on the set $\{g_{\log N}u_{Y}\Z^{d}_{\prim}: 0 \le \|Y\| \le 1\}$. It is well known that $$\eta_N \longrightarrow \mu_{d},$$ as $N \rightarrow \infty$, where the convergence, as before, is in the weak-$*$ topology. This seems to date to Rogers~\cite[p.250, (4)]{Rogers}, who claims the result (without proof) (see Rogers~\cite[Chapter 4]{Rogers:packing}, for a proof of an averaged version). We refer the reader to Kleinbock-Margulis~\cite{KM} where a stronger statement, with a rate of convergence is proved. We note that, Zagier's theorem, used in the introduction also comes with a rate, however the rate of convergence in these equidistribution statements does not shed additional light on the \EST distribution. Let $\chi_{k}$ denote the indicator function of the set
$$\left\{\Lambda \in X~:~ \#\left(\Lambda_{\prim} \cap H_{A, c} = k\right)\right\}.$$The functions $\chi_{k}$ can be approximated by continuous functions with compact support on the space of lattices $X_{d+1}$, so we have, as before,

$$P(\est_{m \times n} (A, c, N) = k) = \eta_N(\chi_K) \underset{N \to \infty}{\longrightarrow} \mu_{d}(\chi_K).$$

%
%



\subsection{Diophantine approximation on curves}

\noindent To obtain \EST and Kesten distributions for Diophantine approximation on curves, we follow the same procedure above and use the following equidistribution theorem for expanding translates of curves due to N. Shah \cite{Shah1}.

\begin{Theorem}\label{Shah1}
Let $\phi : [a, b] \to \R^{n}$ be an analytic curve whose image is not contained in a proper affine subspace. Let $\Gamma$ be a lattice in $G$. Then for any $x_0 \in G/\Gamma$ and any bounded continuous function $f$ on $G/\Gamma$,
\begin{equation}\label{eq:Shah1}
\lim_{t \to \infty}\frac{1}{b-a} \int_{a}^{b} f(g_t u(\phi(s))x_0)ds = \int_{G/\Gamma}f d\mu.
\end{equation}
\end{Theorem}

\section{Equidistribution on strata}\label{sec:translation}
\subsection{Almost everywhere equidistribution} We prove Theorems~\ref{theorem:estk:flat} and~\ref{theorem:estk:veech}, using Theorem~\ref{theorem:est} and known equidistribution results on the space of lattices. Here, the equivariant assignment is given by $$S \longmapsto \sum_{z \in \La_S} \delta_z,$$ the counting measure on $\La_S$. In the notation of \S\ref{sec:equivariant}, we denote this assignment $\nu$.

We are interested in the Lebesgue measure of the set of $\theta \in [0, 2\pi)$ so that $$\# \left(r_{\theta} \La_S \cap H_{A, c, N}\right) = k$$ (or $R_{A,N}$). Applying $g_{\log N}$, and using equivariance, we rewrite this as $$\# \left( \La_{g_{\log N} r_{\theta}S} \cap H_{A, c}\right) = k$$ (respectively $R_{A}$).

Let $\eta_N(S)$ denote the uniform probability measure $\frac{d\theta}{2\pi}$ on the curve $$\{g_{\log N} r_{\theta}S: 0 \le \theta < 2\pi\} \subset \Omega_g,$$ and let $\nu_n = \nu$ denote the equivariant assignment. Thus, we we are in a position to apply Theorem~\ref{theorem:est}, with $\mathcal R = H_{A,c}$ (or $R_{A}$). Theorem~\ref{theorem:estk:flat} then follows from the aforementioned equidistribution result of Nevo which states that $\eta_N(S_0) \rightarrow \mu_{\hh}$ for $\mu_{\hh}$-a.e. $S_0 \in \hh$.

\subsection{Lattice surfaces} For Theorem~\ref{theorem:estk:veech}, we restrict our universe to the subset $\SL(2, \R) S_0 \cong \SL(2, \R)/\Gamma$, where $\Gamma = \SL(X_0, \omega_0)$. Now, the sequence of measures can be viewed as the measures $\frac{d\theta}{2\pi}$ supported on large circles $$\{g_{\log N} r_{\theta} \Gamma: 0 \le \theta < 2\pi\}.$$ By, for example, Dani-Smillie~\cite{DaniSmillie}, the limiting measure is the Haar probability measure on $\SL(2,\R)/\Gamma$, yielding our result.\qed\medskip

\subsection{Other equivariant assignments} We note that there are many other equivariant assignments (see~\cite[\S2]{EskinMasur} and ~\cite{Veech:siegel}) which can be studied in the context of translation surfaces. Our results, of course, apply to all such assignments.


\end{document}